\documentclass[12pt]{amsart}
\usepackage{amsmath}
\usepackage{amstext}
\usepackage{amsfonts}
\usepackage{amssymb}
\usepackage{amsthm}
\usepackage{amsrefs}
\usepackage{bbm}

\usepackage{microtype}
\usepackage[colorlinks=true, allcolors=blue]{hyperref}

\theoremstyle{plain}
\newtheorem{thm}{Theorem}[section]

\newtheorem{cor}[thm]{Corollary}
\newtheorem{prop}[thm]{Proposition}
\newtheorem{lem}[thm]{Lemma}

\theoremstyle{definition}
\newtheorem{defn}[thm]{Definition}
\newtheorem{exam}[thm]{Example}
\newtheorem{rem}[thm]{Remark}
\newtheorem*{question*}{Question}

\newtheorem{thmintro}{Theorem}

\newcommand{\bF}{{\mathbb{F}}}

\newcommand{\bR}{{\mathbb{R}}}

\newcommand{\bZ}{{\mathbb{Z}}}

\newcommand{\A}{{\mathcal{A}}}

\renewcommand{\S}{{\mathcal{S}}}

\newcommand{\U}{{\mathcal{U}}}
\newcommand{\V}{{\mathcal{V}}}
\newcommand{\W}{{\mathcal{W}}}

\newcommand{\rC}{{\mathrm{C}}}

\renewcommand{\phi}{\varphi}

\newcommand{\ca}{\mathrm{C}^*}

\newcommand{\fb}{\partial_F}

\begin{document}
\title[Simplicity of reduced crossed products]{Simplicity of reduced crossed products}

\author[T. Bray]{Thomas Bray}
\address{Department of Pure Mathematics\\ University of Waterloo\\
Waterloo, Ontario \; N2L 3G1 \\Canada}
\email{tgrbray@uwaterloo.ca}

\author[M. Kennedy]{Matthew Kennedy}
\address{Department of Pure Mathematics\\ University of Waterloo\\
Waterloo, Ontario \; N2L 3G1 \\Canada}
\email{matt.kennedy@uwaterloo.ca}

\begin{abstract}
We characterize the simplicity of reduced crossed product C*-algebras in terms of stabilizer subgroups. Specifically, we prove that if $G$ is a countable group and $X$ is a minimal $G$-flow, then the reduced crossed product C*-algebra $\rC(X) \times_\lambda G$ is simple if and only if there is a point in $X$ with a C*-simple stabilizer subgroup. Further, these conditions are equivalent to a generic point in $X$ having a C*-simple stabilizer subgroup. We also provide an example demonstrating that this result does not extend to uncountable groups. This completely resolves a question of Ozawa.
\end{abstract}

\thanks{First author partially supported by an OGS Graduate Scholarship (Ontario, Canada). Second author partially supported by an NSERC Discovery Grant (Canada). }
\maketitle

\section{Introduction}

Let $G$ be a discrete group and let $X$ be a $G$-flow, i.e. a compact Hausdorff space equipped with an action of $G$ by homeomorphisms. The corresponding reduced crossed product $\rC(X) \times_\lambda G$ is a C*-algebra that encodes a great deal of information about the action of $G$ on $X$. Beginning essentially with von Neumann, a great deal of literature has been generated in an attempt to understand the precise relationship between the C*-algebraic structure of the reduced crossed product and the dynamical structure of the underlying action.

Perhaps the most fundamental problem along these lines concerns the simplicity of reduced crossed products. Specifically, the problem asks for a characterization of the simplicity of $\rC(X) \times_\lambda G$ in terms of the dynamical structure of the $G$-action on $X$. Recall that a C*-algebra is simple if it has no non-trivial closed two-sided ideals. 

An obvious necessary condition for the simplicity of the reduced crossed product is the minimality of $X$, i.e. the absence of any non-trivial proper $G$-subflows, since a $G$-invariant non-empty proper open subset $U \subseteq X$ yields the non-trivial closed two-sided ideal $\rC_0(U) \times_\lambda G$ of $\rC(X) \times_\lambda G$. 

Sufficient conditions for the simplicity of the reduced crossed product have been known for some time. Elliott \cite{Ell1980} proved that $\rC(X) \times_\lambda G$ is simple if $X$ is minimal and topologically free, meaning that the set of fixed points of every non-trivial element in $G$ has empty interior. This generalized earlier results of Effros-Hahn \cite{EH1967}. Moreover, for amenable $G$, the converse was proved by Kawamura-Tomiyama \cite{KT1990} (see also subsequent work of Archbold-Spielberg \cite{AS1994}). 

However, if $G$ is non-amenable and $X$ is minimal, then topological freeness of $X$ is generally far from necessary for the simplicity of the reduced crossed product. Indeed, if $X$ is the trivial $G$-flow consisting of a single point, then the reduced crossed product coincides with the reduced C*-algebra $\ca_\lambda(G)$. For non-amenable $G$, although $X$ is never topologically free, it turns out that $\ca_\lambda(G)$ is frequently simple. A group with this property is said to be {\em C*-simple}.

The first example of a C*-simple group was discovered by Powers \cite{Pow1975}. He showed that the free group on two generators is C*-simple. Subsequently, many more examples of C*-simple groups were found using variants of Powers' original argument. 

The first complete characterization of C*-simplicity was established by Kalantar and the second author \cite{KK2017}*{Theorem 1.5} utilizing boundary-theoretic techniques. Specifically, they proved that $G$ is C*-simple if and only if the action of $G$ on its Furstenberg boundary $\fb G$ is (topologically) free. In further work, the second author established an ``intrinsic'' group-theoretic characterization of C*-simplicity \cite{Ken2020}*{Theorem 1.1}. Specifically, he proved that $G$ is C*-simple if and only if it has no confined amenable subgroups (see Section \ref{sec:chabauty}), which is equivalent to the absence of any non-trivial amenable uniformly recurrent subgroup (see Section \ref{sec:urs}). 

The first complete characterization of the simplicity of reduced crossed products was established by Kawabe \cite{Kaw2017}*{Theorem 1.7}. Generalizing \cite{Ken2020}, he proved that if $X$ is a minimal $G$-flow, then the reduced crossed product $\rC(X) \times_\lambda G$ is simple if and only if it has no confined amenable subgroups contained in the stabilizer subgroup of a point in $X$, which is equivalent to the absence of any non-trivial amenable uniformly recurrent subgroup dominated by the stabilizer uniformly recurrent subgroup of the action with respect to the partial order on uniformly recurrent subgroups (see Section \ref{sec:urs}). 

In practice, Kawabe's criterion can be difficult to verify. In order to establish the simplicity of the reduced crossed product, it is typically much easier to verify a remarkably simple sufficient condition obtained by Ozawa \cite{Oza2014}*{Theorem 14} (see also \cite{BKKO2017}*{Proposition 7.6}). Specifically, he proved that if the stabilizer subgroup of any point in $X$ is C*-simple, then $\rC(X) \times_\lambda G$ is simple. He further asked if the converse holds.

\begin{question*}[Ozawa]
Let $G$ be a discrete group and let $X$ be a minimal $G$-flow. Is the reduced crossed product $\rC(X) \times_\lambda G$ simple if and only if there is a point $x \in X$ such that the stabilizer subgroup $G_x$ is C*-simple?
\end{question*}

Recently, Kalantar-Hartman made significant progress on Ozawa's question \cite{HK2026}*{Theorem A}. They proved that if $G$ is countable and $X$ is a minimal $G$-flow such that the reduced crossed product $\rC(X) \times_\lambda G$ is simple, then there is a point $x \in X$ such that the stabilizer subgroup $G_x$ has trivial amenable radical. Although the triviality of the amenable radical is not equivalent to C*-simplicity in general, it is known to be equivalent for a large class of groups.

In this paper, we will completely resolve Ozawa's question. For countable groups, we will prove that Ozawa's question has an affirmative answer. Strikingly, however, we will demonstrate that the answer to Ozawa's question can be negative for an uncountable discrete group. As far as the authors are aware, this is the first time that this kind of divergence between the behaviour of countable and uncountable discrete groups has been observed in this setting.

The following result is the main result in this paper. It provides an affirmative answer to Ozawa's question for countable groups. The result is proved in Section \ref{sec:urs}.

\begin{thmintro} \label{thmintro:main}
Let $G$ be a countable group and let $X$ be a minimal $G$-flow. The following are equivalent:
\begin{enumerate}
\item The reduced crossed product $\rC(X) \times_\lambda G$ is simple;
\item There is a point $x \in X$ with C*-simple stabilizer subgroup $G_x$;
\item A generic point $x \in X$ has a C*-simple stabilizer subgroup $G_x$, meaning that the subset
\[
\{x \in X : G_x \text{ is C*-simple} \}
\]
is comeager in $X$.
\end{enumerate}
\end{thmintro}

We demonstrate that the answer to Ozawa's question can be negative for an uncountable discrete group in Example \ref{exam:uncountable}.

In addition to this introduction, this paper has two sections. In Section~\ref{sec:chabauty}, we consider the Chabauty space of subgroups of a discrete group. We prove that for a countable group, the subset of C*-simple subgroups is $G_\delta$ in the Chabauty space. In Section~\ref{sec:urs}, we consider uniformly recurrent subgroups associated to minimal flows. We prove the existence of C*-simple subgroups in certain uniformly recurrent subgroups arising in this way when the reduced crossed product is simple. This is the key fact needed to prove the main result. Finally, we provide an example demonstrating that the answer to Ozawa's question can be negative for an uncountable discrete group.

\subsection*{Acknowledgements}

The authors are grateful to Adrien Le Boudec, Shirly Geffen and Camila Sehnem for helpful comments. After the completion and circulation of this manuscript, we learned from Paolo Boldrini that he had independently obtained a proof of the main theorem.

\subsection*{AI use statement}

No AI was utilized in the preparation of this manuscript.

\section{The Chabauty space of subgroups} \label{sec:chabauty}

Let $G$ be a discrete group and let $\operatorname{Sub}(G)$ be the set of subgroups of $G$, equipped with the Chabauty topology and the conjugation action of $G$. If $\operatorname{Sub}(G)$ is identified with a subset of $\{0,1\}^G$, then the Chabauty topology coincides with the relative product topology. In particular, $\operatorname{Sub}(G)$ is compact, and hence is a $G$-flow. Note that if, in addition, $G$ is countable, then $\operatorname{Sub}(G)$ is metrizable.

A subgroup $H \in \operatorname{Sub}(G)$ is {\em $G$-confined} if the closure of its $G$-orbit does not contain the trivial subgroup. For a non-empty finite subset $F \subseteq G \setminus \{e\}$, define $\V_F \subseteq \operatorname{Sub}(G)$ by
\[
\V_F = \{H \in \operatorname{Sub}(G) : H \cap F = \emptyset \}.
\]
Then the family $\{\V_F\}_F$ for $F$ as above is a neighborhood basis of the trivial subgroup. It follows that $H \in \operatorname{Sub}(G)$ is confined if and only if there is $\V_F$ such that $g H g^{-1} \notin \V_F$ for all $g \in G$, or equivalently, if and only if $g H g^{-1} \cap F \ne \emptyset$ for all $g \in G$.

\begin{defn} \label{defn:confinement-witness}
Let $G$ be a discrete group. For a subgroup $H \in \operatorname{Sub}(G)$, we will say that a non-empty finite subset $F \subseteq G \setminus \{e\}$ is a {\em $G$-confinement witness for $H$} if $gHg^{-1} \cap F \ne \emptyset$ for all $g \in G$. 
\end{defn}

\begin{rem} \label{rem:confinement-witness}
It follows from the above discussion that a subgroup $H \in \operatorname{Sub}(G)$ is $G$-confined if and only if it has some $G$-confinement witness.
\end{rem}

\begin{lem} \label{lem:confinement-witness}
Let $G$ be a discrete group and let $F \subseteq G \setminus \{e\}$ be a non-empty finite subset. Let $\S \subseteq \operatorname{Sub}(G)$ be a subset such that $F$ is a $G$-confinement witness for every subgroup in $\S$. Then $F$ is a $G$-confinement witness for every subgroup in the closure $\overline{\S}$ of $\S$.
\end{lem}

\begin{proof}
Let $(H_i)_i$ be a net of subgroups in $\S$ converging to a subgroup $H \in \overline{\S}$. For $g \in G$, $g H_i g^{-1} \cap F \ne \emptyset$. Equivalently, letting $\V_F \subseteq \operatorname{Sub}(G)$ denote the neighborhood of the trivial subgroup corresponding to $F$, $g H_i g^{-1} \notin \V_F$. Since $\V_F$ is clopen and $\lim g H_i g^{-1} = g H g^{-1}$, it follows that $g H g^{-1} \notin \V_F$. Equivalently, $g H g^{-1} \cap F \ne \emptyset$. Since $g \in G$ was arbitrary, it follows that $F$ is a $G$-confinement witness for $H$.
\end{proof}

\begin{lem} \label{lem:orbit}
Let $G$ be a group and let $F \subseteq G \setminus \{e\}$ be a non-empty finite subset. If $F$ is a $G$-confinement witness for a subgroup $H \in \operatorname{Sub}(G)$, then $F$ is a $G$-confinement witness for every subgroup in the closure of the $G$-orbit of $H$.
\end{lem}

\begin{proof}
If $F$ is a $G$-confinement witness for $H$, then clearly $F$ is a $G$-confinement witness for every subgroup in the $G$-orbit of $H$. The result now follows from Lemma \ref{lem:confinement-witness}.
\end{proof}

\begin{defn} \label{defn:witnesses}
Let $G$ be a discrete group. For a non-empty finite subset $F \subseteq G \setminus \{e\}$, we will let $\W_F \subseteq \operatorname{Sub}(G)$ denote the set of subgroups $H \leq G$ with the property that $F \cap H$ is an $H$-confinement witness for some amenable subgroup of $H$.
\end{defn}

\begin{lem} \label{lem:confined-amenable-witnessed-closed}
Let $G$ be a discrete group and let $F \subseteq G \setminus \{e\}$ be a non-empty finite subset. The set $\W_F$ is closed in the Chabauty topology on $\operatorname{Sub}(G)$.
\end{lem}

\begin{proof}
Let $(H_i)_i$ be a net of subgroups in $\W_F$ converging to a subgroup $H \in \operatorname{Sub}(G)$. For each $i$, let $K_i \leq H_i$ be an amenable subgroup such that $F \cap H_i$ is an $H_i$-confinement witness for $K_i$. By passing to a subnet, we can assume that the net $(K_i)_i$ also converges to a subgroup $K \in \operatorname{Sub}(G)$. Note that $K \leq H$. Furthermore, since the subset of $\operatorname{Sub}(G)$ consisting of amenable subgroups is closed in the Chabauty topology, $K$ is amenable.

The proof that $F \cap H$ is an $H$-confinement witness for $K$ is similar to the proof of Lemma \ref{lem:confinement-witness}. Fix $h \in H$. Then $h \in H_i$ eventually, which implies that $h K_i h^{-1} \cap F \ne \emptyset$ eventually, since $F \cap H_i$ is an $H_i$-confinement witness for $K_i$. Equivalently, letting $\V_F \subseteq \operatorname{Sub}(G)$ denote the neighborhood of the trivial subgroup corresponding to $F$, $h K_i h^{-1} \notin \V_F$. Since $\V_F$ is clopen and $\lim h K_i h^{-1} = hKh^{-1}$, it follows that $h K h^{-1} \notin \V_F$. Equivalently, $h K h^{-1} \cap F \ne \emptyset$. Since $h \in H$ was arbitrary, it follows that $F \cap H$ is an $H$-confinement witness for $K$. Hence $H \in \W_F$.
\end{proof}

For a discrete group $G$, the set of C*-simple subgroups of $G$ is not necessarily closed nor open in the Chabauty topology on $\operatorname{Sub}(G)$, as the following example demonstrates.

\begin{exam}
Let $G = \bF_2$, where $\bF_2$ denotes the free group over $\{a,b\}$. 

To see that the subset of C*-simple groups is not closed in the Chabauty topology on $\operatorname{Sub}(G)$, let $H = \langle a \rangle \in \operatorname{Sub}(G)$ and for $n \geq 1$, let $H_n = \langle a, b^n a b^{-n} \rangle \in \operatorname{Sub}(G)$. Then $H \cong \bZ$ is not C*-simple since, e.g. it is abelian, while each $H_n$ is C*-simple, since $G_n \cong \bF_2$ and $\bF_2$ is C*-simple \cite{Pow1975}. However, it is easy to verify that $\lim H_n = H$ in the Chabauty topology on $\operatorname{Sub}(G)$.

To see that the subset of C*-simple groups is not open in the Chabauty topology, for $n \geq 1$, let $K_n = \langle a^n \rangle \in \operatorname{Sub}(G)$. Then it is easy to verify that $\lim K_n = \{e\}$ in the Chabauty topology on $\operatorname{Sub}(G)$.
\end{exam}

One of the main results in \cite{Ken2020}*{Theorem 1.1} is that a discrete group is C*-simple if and only if it has no amenable confined subgroups. Utilizing this characterization, we now prove that for a countable group $G$, the set of C*-simple subgroups of $G$ is $G_\delta$ in the Chabauty topology on $\operatorname{Sub}(G)$.

\begin{prop} \label{prop:g-delta}
Let $G$ be a countable group. The subset of C*-simple subgroups of $G$ is $G_\delta$ in the Chabauty topology on $\operatorname{Sub}(G)$.
\end{prop}

\begin{proof}
Letting $\W_F$ be as in Definition \ref{defn:witnesses}, \cite{Ken2020}*{Theorem 1.1} and Remark~\ref{rem:confinement-witness} imply that a subgroup $H \leq G$ is C*-simple if and only if $H \notin \W_F$ for every non-empty finite subset $F \subseteq G \setminus \{e\}$. It follows that the subset of C*-simple subgroups of $G$ is given by
\[
\bigcap_F \operatorname{Sub}(G) \setminus \W_F,
\]
where the intersection is taken over non-empty finite subsets $F \subseteq G \setminus \{e\}$. The result now follows from Lemma \ref{lem:confined-amenable-witnessed-closed} and the countability of $G$. 
\end{proof}

\section{Uniformly recurrent subgroups and crossed products} \label{sec:urs}

The notion of a uniformly recurrent subgroup was introduced by Glasner-Weiss \cite{GW2015}. A $G$-subflow $\U \subseteq \operatorname{Sub}(G)$ is a {\em uniformly recurrent subgroup} (or URS for short) of $G$ if it is minimal, i.e. if it admits no proper subflows. A URS $\U$ is {\em amenable} if every subgroup in $\U$ is amenable.

Let $X$ be a $G$-flow. For a point $x \in X$, let $G_x \leq G$ denote the corresponding stabilizer subgroup
\[
G_x = \{g \in G : gx = x\}.
\]
The stabilizer map
\[
X \to \operatorname{Sub}(G) : x \to G_x
\]
is not necessarily continuous (see e.g. \cite{LBMB2018}*{Lemma 2.2}). However, it is upper semicontinuous, and if $G$ is countable, then letting $X_0 \subseteq X$ denote the set of continuity points, $X_0$ is always a dense $G_\delta$ (see \cite{LBMB2018}*{Proposition 2.4}).

For minimal $X$, Glasner-Weiss associated a URS to $X$, called the {\em stabilizer URS for $X$}, which we will denote by $\S_X$. Specifically, $\S_X$ is the closure of the set $\{G_x : x \in X_0\} \subseteq \operatorname{Sub}(G)$ \cite{GW2015}*{Proposition 1.2}. 

There is a natural partial order on the set of URSs of $G$ introduced by Le Boudec-Matte Bon \cite{LBMB2018}*{Definition 2.11}. Specifically, for URSs $\U$ and $\V$ of $G$, $\U \preccurlyeq \V$ if there is $H \in \U$ and $K \in \V$ with $H \leq K$. By \cite{LBMB2018}*{Proposition 2.14}, $\U \preccurlyeq \V$ if and only if for every subgroup $H \in \U$ there is a subgroup $K \in \V$ such that $H \leq K$, and for every subgroup $K' \in \V$ there is a subgroup $H' \in \U$ such that $H' \leq K'$.

Kawabe proved \cite{Kaw2017}*{Theorem 7.4} that for a minimal $G$-flow $X$ with corresponding stability URS $\S_X$ of $G$, there is a unique amenable URS $\A_X$ of $G$ such that $\A_X \preccurlyeq \S_X$ and whenever $\U$ is another amenable URS of $G$ with $\U \preccurlyeq \S_X$, then $\U \preccurlyeq \A_X$.

A particularly important example of a URS obtained in this way is the {\em Furstenberg URS} $\A_{\fb G}$ of $G$. In fact, $\A_{\fb G} = \S_{\fb G}$ \cite{LBMB2018}*{Proposition 2.21}, and if $\U \subseteq \operatorname{Sub}(G)$ is any amenable URS of $G$, then $\U \preccurlyeq \A_{\fb G}$ \cite{LBMB2018}*{Theorem 2.16}.

\begin{lem} \label{lem:furstenberg-urs-confinement-witness}
Let $G$ be a discrete group. For a non-empty finite subset $F \subseteq G \setminus \{e\}$, let $\W_F \subseteq \operatorname{Sub}(G)$ be defined as in Definition \ref{defn:witnesses}. Then for a subgroup $H \in \operatorname{Sub}(G)$, $H \in \W_F$ if and only if $F \cap H$ is an $H$-confinement witness for every subgroup in the Furstenberg URS $\A_{\fb H}$ of $H$.
\end{lem}

\begin{proof}
If every subgroup $K \in \A_{\fb H}$ has $H$-confinement witness $F \cap H$, then $H \in \W_F$ by definition. Conversely, suppose $H \in \W_F$. Then there is an amenable subgroup $K \leq H$ with $H$-confinement witness $F \cap H$. Since the subset of $\operatorname{Sub}(G)$ consisting of amenable subgroups is closed in the Chabauty topology, it follows from Zorn's lemma that the closure of the $H$-orbit of $K$ contains an amenable URS $\U$ of $H$. By Lemma \ref{lem:orbit}, $F \cap H$ is an $H$-confinement witness for every subgroup in $\U$. The result now follows from the fact that $\U \preccurlyeq \A_{\fb H}$.
\end{proof}

\begin{prop} \label{prop:urs}
Let $G$ be a countable group and let $X$ be a minimal $G$-flow such that the reduced crossed product $\rC(X) \times_\lambda G$ is simple. Then every URS $\U$ of $G$ with $\U \preccurlyeq \S_X$ contains a C*-simple subgroup.
\end{prop}

\begin{proof}
Let $\U$ be a URS of $G$ with $\U \preccurlyeq \S_X$. Suppose for the sake of contradiction that no subgroup in $\U$ is C*-simple.

For a non-empty finite subset $F \subseteq G \setminus \{e\}$, let $\W_F \subseteq \operatorname{Sub}(G)$ be as in Definition \ref{defn:witnesses} and let $\U_F = \U \cap \W_F$. Then each $\U_F$ is relatively closed by Lemma \ref{lem:confined-amenable-witnessed-closed}. It follows from \cite{Ken2020}*{Theorem 1.1} and Remark \ref{rem:confinement-witness} that $\U = \cup_F \U_F$, where the union is taken over all non-empty finite subsets $F \subseteq G \setminus \{e\}$. Since $G$ is countable, this is a countable union of closed subsets. Hence by the Baire category theorem, there is a non-empty finite subset $F_0 \subseteq G \setminus \{e\}$ such that $\U_{F_0}$ has non-empty interior.

By the compactness and minimality of $\U$, there is $g_1,\ldots,g_n \in G$ such that
\[
\U = \bigcup_{i=1}^n g_i \U_{F_0} g_i^{-1} = \bigcup_{i=1}^n \U_{g_i^{-1} F_0 g_i},
\]
where we have used the fact that for $g \in G$, $g \U_{F_0} g^{-1} = \U_{g^{-1} F_0 g}$. Letting $E = \cup_{i=1}^n g_i^{-1} F_0 g_i$, it follows that $\U = \U_E$.

Since $\U \preccurlyeq \S_X$, there is $H \in \U$ with $H \leq G_x$ for some $x \in X_0$. Fix $H$ and fix an amenable subgroup $K \in \A_{\fb H}$. Since $H \in \U_E$, Lemma \ref{lem:furstenberg-urs-confinement-witness} implies that $K$ has $H$-confinement witness $E \cap H$.

For $g \in G$, $g \A_{\fb H} g^{-1} = \A_{\fb g H g^{-1}}$, so $gKg^{-1} \in \A_{\fb g H g^{-1}}$. Since $gHg^{-1} \in \U_E$, applying Lemma \ref{lem:furstenberg-urs-confinement-witness} again implies that $gKg^{-1}$ has $gHg^{-1}$-confinement witness $E \cap gHg^{-1}$. In particular, $g K g^{-1} \cap E \ne \emptyset$. Since $g \in G$ was arbitrary, we obtain that $K$ is a confined amenable subgroup of $G$. Since $K \leq H \leq G_x$, it now follows from \cite{Kaw2017}*{Theorem 1.7} that the reduced crossed product $\rC(X) \times_\lambda G$ is not simple, giving a contradiction.
\end{proof}

\begin{rem}
The proof of Proposition~\ref{prop:urs} implies that if $G$ is a countable group and $\U$ is a URS of $G$ such that every subgroup $H \in \U$ has non-trivial Furstenberg URS $\A_{\fb H}$, then there is a finite subset $F \subseteq G \setminus \{e\}$ such that for every $H \in \U$ and every subgroup $K \in \A_{\fb H}$, $K \cap F \ne \emptyset$. In particular, every subgroup $K \in \A_{\fb H}$ is $G$-confined. 
\end{rem}

Let $X$ and $Y$ be compact Hausdorff spaces. Recall that a map $\pi : Y \to X$ is {\em category-preserving} if the preimage under $\pi$ of every comeager subset of $X$ is a comeager subset of $Y$. We will require the following facts about equivariant maps between minimal flows.

\begin{lem} \label{lem:category-preservation}
    Let $G$ be a discrete group, let $X$ and $Y$ be minimal $G$-flows and let $\pi : Y \to X$ be a continuous $G$-equivariant map. Then $\pi$ is category-preserving. Furthermore, $\pi$ maps comeager subsets of $Y$ to comeager subsets of $X$.
\end{lem}

\begin{proof}
    Minimality implies that $\pi$ is surjective and semi-open, meaning that it maps subsets of $Y$ with non-empty interior to subsets of $X$ with non-empty interior (see e.g. \cite{BKKO2017}*{Lemma 3.2}). It follows easily from this that $\pi$ is category-preserving. The final statement is a special case of a result of Noll \cite{Nol1989}*{Section 3}.
\end{proof}

\begin{lem} \label{lem:dense-g-delta}
Let $G$ be a countable group and let $X$ be a minimal $G$-flow with corresponding stabilizer URS $\S_X$. Let $X_0 \subseteq X$ denote the dense $G_\delta$ subset of continuity points for the stabilizer map. Then the subset
\[
\{G_x : x \in X_0\}
\]
is comeager in $\S_X$.
\end{lem}

\begin{proof}
Let $\S_{X,0} = \{G_x : x \in X_0\}$ and let $\tilde{X} \subseteq X \times \S_X$ denote the closure of the set
\[
\tilde{X}_0 := \{(x,G_x) : x \in X_0\}.
\]
Then by \cite{GW2015}*{Proposition 1.2}, $\tilde{X}$ is a minimal $G$-flow with respect to the diagonal action. Let $\pi_1 : \tilde{X} \to X$ denote the projection map onto the first coordinate and $\pi_2 : \tilde{X} \to \S_X$ denote the projection map onto the second coordinate. Then $\tilde{X}_0 = \pi_1^{-1}(X_0)$, since if a net $((x_i, G_{x_i}))_i$ in $\tilde{X}_0$ converges to $(x,H) \in \tilde{X}$ with $x \in X_0$, then the fact that $x$ is a continuity point for the stabilizer map implies that $H = G_x$. Since $\pi_1$ and $\pi_2$ are continuous $G$-equivariant maps between minimal flows and $X_0$ is a dense $G_\delta$ in $X$, it now follows from Lemma~\ref{lem:category-preservation} that $\S_{X,0}=\pi_2(\tilde{X}_0)=\pi_2(\pi_1^{-1}(X_0))$ is comeager in $\S_X$.
\end{proof}

\begin{proof}[Proof of Theorem \ref{thmintro:main}]
The implication (2) $\Rightarrow$ (1) is \cite{Oza2014}*{Theorem 14}, and the implication (3) $\Rightarrow$ (2) is clear.

For the implication (1) $\Rightarrow$ (3), suppose that the reduced crossed product $\rC(X) \times_\lambda G$ is simple. Then Proposition~\ref{prop:urs} implies that the stabilizer URS $\S_X$ of $X$ contains a C*-simple subgroup. It follows from the minimality of $\S_X$ and Proposition~\ref{prop:g-delta} that the set of C*-simple subgroups of $\S_X$ is a dense $G_\delta$ in $\S_X$, and in particular is comeager. Also, by Lemma~\ref{lem:dense-g-delta}, the set
\[
\S_{X,0} := \{G_x : x \in X_0\}
\]
is comeager in $\S_X$. Therefore, the intersection
\[
\S_{X,0}' := \{G_x : x \in X_0,\ G_x \text{ is C*-simple}\}
\]
is comeager in $\S_X$.

Let $\tilde{X} \subseteq X \times \S_X$ be defined as in the proof of Lemma \ref{lem:dense-g-delta}. Recall that $\tilde{X}$ is minimal. Letting $\pi_2 : \tilde{X} \to \S_X$ denote the projection map onto the second coordinate, Lemma~\ref{lem:category-preservation} implies that $\pi_2$ is category-preserving and hence that $\pi_2^{-1}(\S_{X,0}')$ is comeager in $\tilde{X}$. Also, it follows from the proof of Lemma \ref{lem:dense-g-delta} that the set
\[
\tilde{X}_0 := \{(x,G_x) : x \in X_0\}
\]
is comeager in $\tilde{X}$. Therefore, the intersection
\[
\tilde{X}_0' := \{(x,G_x) : x \in X_0,\ G_x \text{ is C*-simple} \}
\]
is also comeager in $\tilde{X}$.

Finally, Letting $\pi_1 : \tilde{X} \to X$ denote the projection map onto the first coordinate, the set
\[
X_0' := \{x \in X_0 : G_x \text{ is C*-simple}\}
\]
is equal to $\pi_1(\tilde{X}_0')$, so it follows from another application of Lemma \ref{lem:category-preservation} that $X_0'$ is comeager in $X$ by Lemma~\ref{lem:category-preservation}.
\end{proof}

Matte Bon-Tsankov \cite{MBT2020}*{Theorem 1.1} proved that every URS of $G$ arises as the stabilizer URS of a minimal $G$-flow. For countable $G$, combining this result with Theorem \ref{thmintro:main} yields the following interesting consequence for the structure of certain URSs of $G$.

\begin{cor} \label{cor:urs}
 Let $G$ be a countable group and let $X$ be a minimal $G$-flow such that the reduced crossed product $\rC(X) \times_\lambda G$ is simple. Let $\U$ be a URS of $G$ with $\U \preccurlyeq \S_X$. Then a generic subgroup in $\U$ is C*-simple, meaning that the set
\[
\{H \in \U : H \text{ is C*-simple} \}
\]
is comeager in $\U$. Consequently, if a URS $\V$ of $G$ contains a C*-simple subgroup, then a generic subgroup of $\V$ is C*-simple, and the same is true of any URS that it dominates.
\end{cor}

\begin{proof}
Let $\U$ be a URS of $G$ with $\U \preccurlyeq \S_X$. Proposition \ref{prop:urs} implies that $\U$ contains a C*-simple subgroup. It follows from the minimality of $\U$ and Proposition \ref{prop:g-delta} that the set of C*-simple subgroups in $\U$ is a dense $G_\delta$ in $\U$. The final statement can be obtained by applying the realization theorem \cite{MBT2020}*{Theorem~1.1} and Theorem \ref{thmintro:main}.
\end{proof}

For the case of countable C*-simple $G$, invoking a result from \cite{BKKO2017} allows us to apply Corollary \ref{cor:urs} to every URS of $G$.

\begin{cor}
Let $G$ be a countable C*-simple group and let $\U$ be a URS of $G$. Then a generic subgroup in $\U$ is C*-simple, meaning that the set
\[
\{H \in \U : H \text{ is C*-simple} \}
\]
is comeager in $\U$.
\end{cor}

\begin{proof}
Let $\U$ be a URS for $G$. Then by \cite{MBT2020}*{Theorem~1.1}, there is a minimal $G$-flow $X$ such that $\U = \S_X$. Since $G$ is C*-simple, \cite{BKKO2017}*{Theorem~7.1} implies that the reduced crossed product $\rC(X) \times_\lambda G$ is simple. The result now follows from Corollary \ref{cor:urs}.
\end{proof}

The following example demonstrates that Theorem~\ref{thmintro:main} can fail for uncountable discrete groups. The second author is grateful to Sven Raum for discussions about $\mathrm{PSL}_2(\bR)$. 

\begin{exam} \label{exam:uncountable}
The Furstenberg boundary of the locally compact group $\mathrm{PSL}_2(\bR)$ is the real projective line $\mathbb{RP}^1$. Letting $G$ be the (uncountable) discretization of $\mathrm{PSL}_2(\bR)$, we obtain a boundary action of $G$ on $\mathbb{RP}^1$ after composing with the identity map $G \to \mathrm{PSL}_2(\bR)$. An element in $G \setminus \{e\}$ has at most two fixed points in $\mathbb{RP}^1$, so the action is topologically free. Hence by \cite{Ell1980}*{Theorem 3.2}, the reduced crossed product $\rC(\mathbb{RP}^1) \times_\lambda G$ is simple.

On the other hand, for the point $\infty \in \mathbb{RP}^1$, the corresponding stabilizer subgroup $G_\infty$ is solvable, and hence not C*-simple. Moreover, the $G$-action on $\mathbb{RP}^1$ is transitive, so the stabilizer of every point in $\mathbb{RP}^1$ is conjugate to $G_\infty$, and therefore not C*-simple.
\end{exam}

\end{document}